\documentclass[12pt]{amsart}
\usepackage{amsmath,amsthm,amssymb,amsfonts}
\usepackage{tikz}
\usepackage{placeins}
\usepackage{needspace}
\usepackage{microtype}
\usepackage[hidelinks]{hyperref}
\hypersetup{pdftitle={The lambda-chain condition for unit interval determinantal fiber rings},pdfauthor={Takayuki Hibi and Sora Miyashita},pdfkeywords={unit interval determinantal facet ideals, special fiber rings, pseudo-Gorenstein}}
\newtheorem{Theorem}{Theorem}[section]
\newtheorem{Lemma}[Theorem]{Lemma}
\newtheorem{Corollary}[Theorem]{Corollary}
\theoremstyle{definition}
\newtheorem{Example}[Theorem]{Example}
\theoremstyle{remark}
\newtheorem{Remark}[Theorem]{Remark}
\newcommand{\Jc}{\mathcal J}
\newcommand{\Fc}{\mathcal F}
\newcommand{\Rc}{\mathcal R}
\newcommand{\mm}{\mathfrak m}
\DeclareMathOperator{\Hilb}{Hilb}
\begin{document}
\title[Unit interval determinantal fiber rings]{The $\lambda$-chain condition for unit interval determinantal fiber rings}
\author[T.~Hibi]{Takayuki Hibi}
\author[S.~Miyashita]{Sora Miyashita}
\address{(Takayuki Hibi) Department of Pure and Applied Mathematics, Graduate School of Information Science and Technology, The University of Osaka, Suita, Osaka 565--0871, Japan}
\email{hibi@math.sci.osaka-u.ac.jp}
\address{(Sora Miyashita) Department of Pure and Applied Mathematics, Graduate School of Information Science and Technology, The University of Osaka, Suita, Osaka 565--0871, Japan}
\email{u804642k@ecs.osaka-u.ac.jp}
\begin{abstract}
It is shown that certain finite partially ordered sets arising from unit interval determinantal facet ideals satisfy the $\lambda$-chain condition. In addition, their applications to commutative algebra are briefly discussed.
\end{abstract}
\maketitle
\thispagestyle{empty}

\section*{Introduction}
A chain is a totally ordered subset of a poset. A maximal chain cannot be enlarged; a maximum chain has the most elements. A finite poset $P$ satisfies the \emph{$\lambda$-chain condition} if every element belongs to a maximum chain; see \cite[p.~308]{EC1}. It is \emph{pure} if all maximal chains have the same length. Purity implies the $\lambda$-chain condition, which characterizes pseudo-Gorenstein Hibi rings \cite[Theorem~2.1]{pseudo}.

Fix $m\geq2$ and intervals $C_i=[a_i,b_i]\subseteq[n]$ such that $a_1<\cdots<a_s$, $b_1<\cdots<b_s$, and $|C_i|\geq m$. Set
\[
L=\bigcup_{i=1}^s\binom{C_i}{m},\qquad \rho_i=|C_i\cap C_{i+1}|\quad(1\leq i<s).
\]
Identify each $m$-subset with its increasing tuple. With the componentwise order, $\binom{[n]}m$ is a distributive lattice whose meet and join are the componentwise minimum and maximum. If $I\in\binom{C_i}m$ and $J\in\binom{C_j}m$ with $i\leq j$, then $I\wedge J\in\binom{C_i}m$ and $I\vee J\in\binom{C_j}m$. Thus $L$ is a distributive sublattice; see \cite[Lemma~3.4]{ALL}. Let $P$ be its subposet of join-irreducibles, the elements with exactly one lower cover. Then $L\cong\Jc(P)$, where $\Jc(P)$ consists of the downward closed subsets of $P$, ordered by inclusion.

\begin{Theorem}\label{lambda-general}
Suppose that $\rho_i\leq m$ or $\rho_i\geq2m-1$ for every $i<s$. Then $P$ satisfies the $\lambda$-chain condition. Moreover, $P$ is pure if and only if $\rho_i\leq2m-1$ for every $i<s$.
\end{Theorem}

Let $\Delta_I$ denote the maximal minor on columns $I$ of a generic $m\times n$ matrix $X=(x_{ij})$ over a field $K$. The complex with facet set $L$ is unit interval, with maximal cliques $C_i$ and determinantal facet ideal $J=(\Delta_I:I\in L)$. Since $J$ is generated in one degree, its special fiber ring is
\[
R=K[\Delta_I:I\in L]\cong\Fc(J)=\bigoplus_{q\geq0}J^q/\mm J^q,
\]
where $\mm=(x_{ij})$ and each generating minor of $R$ has degree one.

A diagonal SAGBI degeneration of $R$ gives a Hibi ring with the same $h$-vector \cite{ALL,LS}. Theorem~\ref{lambda-general} therefore has the following consequences.

\begin{Corollary}\label{pseudo-general}
Suppose that $\rho_i\leq m$ or $\rho_i\geq2m-1$ for every $i<s$. Then the special fiber ring $R$ is pseudo-Gorenstein.
\end{Corollary}

\begin{Corollary}\label{equivalences-general}
Suppose that $\rho_i\leq m$ or $\rho_i\geq2m-1$ for every $i<s$. Then the following are equivalent: $R$ is Gorenstein; $R$ is level; $R$ is nearly Gorenstein; $\rho_i\leq2m-1$ for every $i<s$.
\end{Corollary}

For $m=2$, the hypothesis is automatic, so we recover the results for compatible sublattices of the lattice of pairs (Corollary~\ref{pairs}). We prove Theorem~\ref{lambda-general} in Section~1 and discuss its algebraic consequences in Section~2. Intersection conditions are vacuous when $s=1$.

\section*{Acknowledgments}
The first author was partially supported by JSPS KAKENHI Grant Number JP25K07100. The second author was supported by JSPS KAKENHI Grant Number JP25KJ1744. They used ChatGPT to help develop the generalization and check the proofs, and take full responsibility for the content.

\section{Partially ordered sets}
We first describe the join-irreducible elements and their order. We then bound the number of elements in a chain and construct chains attaining this bound.

\begin{Lemma}\label{large}
If $\rho_i\geq2m-1$ for every $i<s$, then $P$ satisfies the $\lambda$-chain condition. Moreover, $P$ is pure if and only if $\rho_i=2m-1$ for every $i<s$.
\end{Lemma}
\begin{proof}
Set $A=a_1$ and $B=b_s$. If $L$ has one element, then $P$ is empty and the assertions are immediate. Otherwise, for $A+m\leq j\leq B$, let $C_{k(j)}$ be the first interval containing $j$, with left endpoint $\alpha(j)$. We claim that the join-irreducible elements are
\begin{equation}\label{tuples}
z_{r,j}=(\alpha(j),\ldots,\alpha(j)+m-r-1,j-r+1,\ldots,j),\qquad 1\leq r\leq m.
\end{equation}
The final segment has length $r$ and last entry $j$; the first is empty for $r=m$. These tuples are distinct, since $j\geq\alpha(j)+m$. For $m=2$, the two rows are $z_{1,j}=(\alpha(j),j)$ and $z_{2,j}=(j-1,j)$.

A tuple in $\binom{[a,b]}m$ is join-irreducible precisely when exactly one entry can be decreased by one while remaining in the lattice. This proves \eqref{tuples} for the first interval. Adjoin the remaining interval lattices successively. For $k\geq2$, the minimum of $\binom{C_k}m$ lies in $\binom{C_{k-1}}m$, and every new tuple ends at $j>b_{k-1}$, so no new tuple lies below an old one. Thus the old join-irreducibles remain unchanged. New tuples that are not join-irreducible in $\binom{C_k}m$ remain joins of two smaller tuples. A new tuple $z_{r,j}$ has a unique lower cover in $\binom{C_k}m$, obtained by decreasing entry $m-r+1$ by one. Every old $y<z_{r,j}$ has $y_m<j$, hence $y_{m-r+1}\leq j-r$, so lies below this cover. Induction proves \eqref{tuples}.

Arrange $z_{r,j}$ in row $r$ and column $j$. We call the relations $z_{r,j}<z_{r+1,j}$ and $z_{r,j}<z_{r,j+1}$, whenever defined, \emph{grid steps}. A \emph{grid path} is a chain whose successive elements are related by grid steps. The order on $P$ is generated by these steps and the additional relations $d_k:=z_{m,a_k+m-1}<u_k:=z_{1,b_{k-1}+1}$ for $2\leq k\leq s$.
The elements $d_k$ and $u_k$ are defined since $A+m\leq a_k+m-1\leq b_{k-1}$, and the relations follow from their coordinates. Conversely, if $z_{r,j}<z_{q,\ell}$, then $j\leq\ell$. For $r\leq q$, a grid path suffices. For $r>q$, the $(m-q)$th entries give $j-q\leq\alpha(\ell)+m-q-1$. Thus $\alpha(j)+m\leq j\leq\alpha(\ell)+m-1$, so $k(j)<k(\ell)$. With $k=k(\ell)$, we have $z_{r,j}\leq d_k<u_k\leq z_{q,\ell}$, with grid paths on either side.

Assign $z_{r,j}$ the label $\varphi(z_{r,j})=j-A-m+r$, one plus the number of steps in any grid path from $z_{1,A+m}$ to $z_{r,j}$. These labels range from $1$ to $B-A$ and increase by one along grid steps. For the additional relations $d_k<u_k$, the intersection hypothesis gives
\begin{equation}\label{gap}
\varphi(u_k)-\varphi(d_k)=\rho_{k-1}-2m+2\geq1.
\end{equation}
Thus $\varphi$ is strictly increasing, so every chain has at most $B-A$ elements. Every element of $P$ lies on a grid path from the minimum $z_{1,A+m}$ to the maximum $z_{m,B}$ with $B-A$ elements, proving the $\lambda$-chain condition.

To determine when $P$ is pure, we examine its covers. If all $\rho_i=2m-1$, every generating relation, and hence every cover, increases the label by one, so all maximal chains have $B-A$ elements.

Otherwise, choose $k$ with $\rho_{k-1}>2m-1$. We claim that $d_k<u_k$ is a cover in $P$. The first $m-1$ entries of both endpoints are $a_k,\ldots,a_k+m-2$. If $d_k\leq z_{r,j}\leq u_k$ and $r\geq2$, its last two entries are consecutive, forcing $z_{r,j}=d_k$. If $r=1$, then $\alpha(j)=a_k$, hence $j>b_{k-1}$ and $z_{r,j}=u_k$. Grid steps are also covers since their labels differ by one. Take grid paths from $z_{1,A+m}$ to $d_k$ and from $u_k$ to $z_{m,B}$, and connect them by this cover. The resulting chain is maximal and has fewer than $B-A$ elements, since $\varphi(u_k)-\varphi(d_k)>1$ by \eqref{gap}.\qedhere
\end{proof}

\begin{proof}[Proof of Theorem~\ref{lambda-general}]
Suppose $\rho_k\leq m$, and let $L^-$ and $L^+$ be the unions of the interval lattices over $i\leq k$ and $i>k$, respectively. Since $b_k-m+1\leq a_{k+1}$, we have $\max L^-\leq\min L^+$. Let $P^-$ and $P^+$ be their join-irreducible posets, and let $\oplus$ denote ordinal sum, which places its first summand entirely below its second. If $\rho_k=m$, the endpoints coincide and $P\cong P^-\oplus P^+$. Otherwise, $P\cong P^-\oplus\{e\}\oplus P^+$, where $e=\min L^+$ is an additional join-irreducible element.

Maximal chains in an ordinal sum are concatenations of maximal chains in its summands, so the $\lambda$-chain condition and purity each reduce to the summands. Splitting at each $\rho_i\leq m$ and applying Lemma~\ref{large} completes the proof.
\end{proof}

\section{Commutative Algebra}
For a standard graded Cohen--Macaulay algebra $A$ of dimension $d$, write $\Hilb_A(t)=h_A(t)/(1-t)^d$, where $h_A(t)=h_0+\cdots+h_st^s$ and $h_s\ne0$. The algebra is \emph{pseudo-Gorenstein} if $h_s=1$, \emph{level} if its canonical module is generated in a single degree, and \emph{nearly Gorenstein} if the trace ideal of its canonical module contains the homogeneous maximal ideal; see \cite{pseudo,nearlyGor}.

\Needspace{6\baselineskip}
\begin{proof}[Proof of Corollaries~\ref{pseudo-general} and~\ref{equivalences-general}]
By \cite[Theorem~4.3 and Corollary~4.7(iv)]{ALL}, $R$ is a normal Cohen--Macaulay domain whose generating minors form a diagonal SAGBI basis. The initial algebra is graded isomorphic to the Hibi ring $K[L]$ \cite[Proposition~3.2]{LS}. Both have dimension $|P|+1$ and the same $h$-vector. By Theorem~\ref{lambda-general} and \cite[Theorem~2.1]{pseudo}, $R$ is pseudo-Gorenstein.

By Stanley's Hilbert-series criterion \cite[Theorem~4.4]{Stanley}, the Cohen--Macaulay domains $R$ and $K[L]$ are Gorenstein simultaneously. The latter is Gorenstein exactly when $P$ is pure \cite{Hibi87}. Theorem~\ref{lambda-general} therefore gives the intersection criterion.

If $R$ is level, pseudo-Gorensteinness implies that its canonical module is cyclic, so $R$ is Gorenstein. Nearly Gorensteinness gives the same conclusion by \cite[Theorem~A]{Miyashita} when $\dim R\geq2$. If $\dim R=1$, then $P$ is empty and $R\cong K[t]$. The converses are immediate.
\end{proof}

Let $L_n=\{p_{ij}:1\leq i<j\leq n\}$ with the componentwise order. A sublattice $L\subseteq L_n$ is \emph{compatible} if $p_{ij}\in L$ and $i\leq k<\ell\leq j$ imply $p_{k\ell}\in L$. Set $\Rc_K[L]=K[x_i y_j-x_j y_i:p_{ij}\in L]$.

\begin{Corollary}\label{pairs}
Let $L\subseteq L_n$ be compatible and let $L\cong\Jc(P)$. Then $P$ satisfies the $\lambda$-chain condition and $\Rc_K[L]$ is pseudo-Gorenstein. Moreover, for $\Rc_K[L]$, Gorensteinness, levelness and nearly Gorensteinness are equivalent.
\end{Corollary}
\begin{proof}
Deleting isolated vertices and relabeling in increasing order preserves $L$ and $\Rc_K[L]$ up to isomorphism. The graph with edges $\{i,j\}$ for $p_{ij}\in L$ has interval maximal cliques \cite[Corollary~3.6]{BH}. Apply Theorem~\ref{lambda-general} and its corollaries with $m=2$.
\end{proof}

\begin{Example}\label{original-example}
The compatible lattice $L\subseteq L_7$ in Figure~\ref{fig-original} is defined by the intervals $[1,5]$, $[2,6]$, and $[4,7]$. Since $\rho_1=4$ and $\rho_2=3$, $\Rc_K[L]$ is pseudo-Gorenstein but not Gorenstein. The shaded vertices form $P$: each has exactly one lower cover in $L$, and their order is inherited from $L$. The bold \emph{fundamental chain} of $L$ increases the second index whenever possible. Its intersection with $P$ is the first chain below:
\[
\begin{gathered} p_{13}<p_{14}<p_{15}<p_{26}<p_{47}<p_{67},\\ p_{13}<p_{23}<p_{26}<p_{47}<p_{67}. \end{gathered}
\]
The first is a maximum chain with six elements; the second is a maximal chain with five. Thus $P$ is not pure.
\end{Example}

\begin{figure}[htbp]
\centering
\begin{tikzpicture}[x=1.1cm,y=.55cm,every node/.style={font=\small,inner sep=1.5pt}]
\begin{scope}
\node[draw,shape=circle] (12) at (0,0) {12};
\node[draw,shape=circle,fill=black!20] (13) at (1,1) {13};
\node[draw,shape=circle,fill=black!20] (14) at (2,2) {14};
\node[draw,shape=circle,fill=black!20] (15) at (3,3) {15};
\node[draw,shape=circle,fill=black!20] (23) at (0,2) {23};
\node[draw,shape=circle] (24) at (1,3) {24};
\node[draw,shape=circle] (25) at (2,4) {25};
\node[draw,shape=circle,fill=black!20] (26) at (3,5) {26};
\node[draw,shape=circle,fill=black!20] (34) at (0,4) {34};
\node[draw,shape=circle] (35) at (1,5) {35};
\node[draw,shape=circle] (36) at (2,6) {36};
\node[draw,shape=circle,fill=black!20] (45) at (0,6) {45};
\node[draw,shape=circle] (46) at (1,7) {46};
\node[draw,shape=circle,fill=black!20] (47) at (2,8) {47};
\node[draw,shape=circle,fill=black!20] (56) at (0,8) {56};
\node[draw,shape=circle] (57) at (1,9) {57};
\node[draw,shape=circle,fill=black!20] (67) at (0,10) {67};
\draw[line width=3] (12)--(13)--(14)--(15)--(25)--(26)--(36)--(46)--(47)--(57)--(67);
\draw (23)--(24)--(25)--(26);
\draw (34)--(35)--(36);
\draw (45)--(46)--(47);
\draw (56)--(57);
\draw (47)--(57)--(67);
\draw (13)--(23);
\draw (14)--(24)--(34);
\draw (25)--(35)--(45);
\draw (26)--(36)--(46)--(56);
\end{scope}
\end{tikzpicture}
\caption{The Hasse diagram of $L$; $ij$ denotes $p_{ij}$.}
\label{fig-original}
\end{figure}
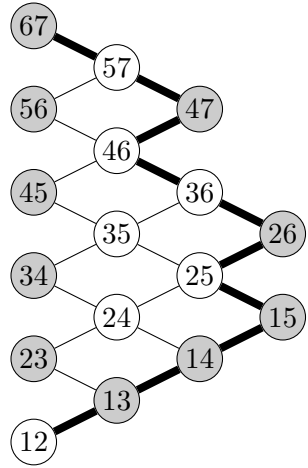

\begin{Remark}
The intersection hypothesis cannot be omitted for higher minors. For $m=3$ and cliques $[1,5]$, $[2,6]$, counting multichains in $L$ gives $h_R(t)=1+6t+10t^2+4t^3$, so $R$ is not pseudo-Gorenstein. This unit interval complex also appears in \cite[Examples~3.6 and~4.2]{ALL}.
\end{Remark}

\FloatBarrier

\end{document}